\documentclass[12pt,a4paper]{amsart}
\usepackage{graphicx} 
\usepackage{latexsym}
\usepackage{amssymb}
\usepackage{verbatim}
\usepackage{url}
\usepackage{hyperref}
\usepackage{anysize}
\marginsize{3cm}{3cm}{3cm}{3cm}

\usepackage{mathtools}
\DeclarePairedDelimiter\ceil{\lceil}{\rceil}
\DeclarePairedDelimiter\floor{\lfloor}{\rfloor}

\DeclareMathOperator{\dist}{dist} 

\newtheorem{thm}{Theorem}
\newtheorem{cor}[thm]{Corollary}
\newtheorem{lem}[thm]{Lemma}

\newtheorem{conjecture}[thm]{Conjecture}

\theoremstyle{definition}

\theoremstyle{remark}

\newcommand{\D}{\Delta}

\date{\today}
\title{On the maximum order of graphs embedded in surfaces}

\author{Eran Nevo}

\address{Department of Mathematics, Ben-Gurion University of the Negev, Israel.}
\email{\texttt{nevoe@cs.bgu.ac.il}}
\author{Guillermo Pineda-Villavicencio}

\address{Centre for Informatics and Applied Optimisation,
Federation University Australia, Australia\\ \and  Department of Mathematics, Ben-Gurion University of the Negev, Israel.}
\email{\texttt{work@guillermo.com.au}}
\author{David R. Wood}
\address{School of Mathematical Sciences, Monash University, Australia.}
\email{\texttt{david.wood@monash.edu}}


\begin{document}
\begin{abstract}
The maximum number of vertices in a graph of maximum degree $\Delta\ge 3$ and fixed diameter $k\ge 2$ is upper bounded by $(1+o(1))(\Delta-1)^{k}$.
  If we restrict our graphs to certain classes, better upper bounds are known. For instance, for the class of trees there is an upper bound of $(2+o(1))(\Delta-1)^{\lfloor k/2\rfloor}$ for a fixed $k$. The main result of this paper is that graphs  embedded in surfaces of bounded Euler genus $g$ behave like trees, in the sense that, for large $\Delta$, such graphs have orders bounded from above by 
  \[\begin{cases} c(g+1)(\Delta-1)^{\lfloor k/2\rfloor} & \text{if $k$ is even}\\
  c(g^{3/2}+1)(\Delta-1)^{\lfloor k/2\rfloor} & \text{if $k$ is odd},
  \end{cases}\]
 where $c$ is an absolute constant. This result represents a qualitative improvement over all previous results, even for planar graphs of odd diameter $k$. With respect to lower bounds, we construct graphs of Euler genus $g$, odd diameter $k$, and order $c(\sqrt{g}+1)(\Delta-1)^{\lfloor k/2\rfloor}$ for some absolute constant $c>0$. Our results answer in the negative a question of Miller and \v{S}ir\'a\v{n} (2005). \end{abstract}

\maketitle
 
 \section{Introduction}

The degree--diameter problem asks for the maximum number of vertices in a graph of  maximum degree $\Delta\ge 3$ and diameter $k\ge 2$. For general graphs the {\it Moore bound}, \[M(\Delta, k):=1+\Delta+\Delta(\Delta-1)+\ldots+\Delta(\Delta-1)^{k-1}= (1+o(1))(\Delta-1)^k (\text{for fixed $k$}),\] provides an upper bound for the order of such a graph.  The well-known de Bruijn graphs provide a lower bound of $\floor{\Delta/2}^{k}$ \cite{Bru46}. For background on this problem the reader is referred to the survey \cite{MS05a}.

If we restrict our attention to particular graph classes, better upper bounds than the Moore bound are possible.  For instance, a well-known result by Jordan \cite{Jordan1869} implies that every tree of maximum degree $\Delta$ and fixed diameter $k$ has at most $(2+o(1))(\Delta-1)^{\lfloor k/2\rfloor}$ vertices. For a graph class $\mathcal{C}$, we define $N(\Delta,k,\mathcal{C})$ to be the maximum order of a graph in $\mathcal{C}$ with  maximum degree $\Delta\ge 3$ and diameter $k\ge 2$. We say $\mathcal{C}$ has {\it small order} if there exists a constant $c$ and a function $f$ such that $N(\Delta,k,\mathcal{C})\le c(\Delta-1)^{\floor{k/2}}$, for all $\Delta\ge f(k)$. The class of trees is a prototype class of small order.

For the class $\mathcal{P}$ of planar graphs, Hell and Seyffarth \cite[Thm.~3.2]{HS93} proved that $N(\Delta,2,\mathcal{P})=\floor{\frac{3}{2}\Delta}+1$ for $\Delta\ge 8$. Fellows {\it et al.}~\cite[Cor.~14]{FHS95} subsequently showed that $N(\Delta,k,\mathcal{P})\le ck\Delta^{\lfloor {k}/{2}\rfloor}$ for every diameter $k$. Notice that this does not prove that $\mathcal{P}$ has small order. Restricting $\mathcal{P}$ to  even diameter assures small order, as shown by Tishchenko's upper bound of  $(\frac{3}{2}+o(1))(\Delta-1)^{{k}/{2}}$, whenever $\Delta \in\Omega(k)$ \cite[Thm.~1.1, Thm.~1.2]{Tis2011}.  Our first contribution is to prove that $N(\Delta,k,\mathcal{P})\le c(\Delta-1)^{\lfloor {k}/{2}\rfloor}$ for $k\ge 2$ and $\Delta \in\Omega(k)$. That is, we show that the class  of planar graphs has small order. 

We now turn our attention to the class $\mathcal{G}_\Sigma$ of graphs embeddable in a surface\footnote{A {\it surface} is a compact (connected) 2-manifold (without boundary). Every surface is homeomorphic to the sphere with $h$ handles or the sphere with $c$ cross-caps \cite[Thm~3.1.3]{MohTho01}. The sphere with $h$ handles has {\it Euler genus} $g:=2h$, while the sphere with $c$ cross-caps has {\it Euler genus } $g:=c$.  For a surface $\Sigma$ and a graph $G$ embedded in $\Sigma$, the (topologically) connected components  of $\Sigma-{G}$  are called {\it faces}. A face homeomorphic to the  open unit disc  is called {\it 2-cell}, and an embedding with only 2-cell faces is called a {\it  2-cell embedding}.  Every face in an embedding is bounded
by a closed walk called a {\it facial walk}.} $\Sigma$ of Euler genus $g$. For diameter 2 graphs, Knor and \v{S}ir\'a\v{n} \cite[Thm.~1, Thm.~2]{KS97} showed that $N(\Delta,2,\mathcal{G}_\Sigma)=N(\Delta,2,\mathcal{P})= \floor{\frac{3}{2}\Delta}+1$, provided $\Delta\in\Omega(g^2)$. \v{S}iagiov\'a and Simanjuntak \cite[Thm.~1]{SR04} proved for all diameters $k$ the upper bound \[N(\Delta,k,\mathcal{G}_\Sigma)\le c(g+1)k(\Delta-1)^{\left\lfloor {k}/{2}\right\rfloor}.\] 

The main contribution of this paper, Theorem~\ref{theo:GralSurface} below, is to show that the class of graphs embedded in a fixed surface $\Sigma$ has small order.
\begin{thm}\label{theo:GralSurface} 
There exists an absolute constant $c$ such that, for every surface $\Sigma$ of Euler genus $g$, 
\[N(\Delta,k,\mathcal{G}_\Sigma)\le\begin{cases}
c(g+1)(\Delta-1)^{\lfloor {k}/{2}\rfloor}& \text{if $k$ is even and $\Delta\ge c(g^{2/3}+1)k$,}\\
c(g^{3/2}+1)(\Delta-1)^{\lfloor {k}/{2}\rfloor}& \text{if $k$ is odd and $\Delta\ge 2k+1$.}
\end{cases}\]
\end{thm} 

We now prove a lower bound on $N(\Delta,k,\mathcal{G}_\Sigma)$ for odd $k\geq3$ (see \cite{FerPin13} for a more complicated construction that gives the same asymptotic lower bound.)\ Let $g$ be the Euler genus of $\Sigma$.  It follows from the Map Colour Theorem \cite[Thm~4.4.5, Thm.~8.3.1]{MohTho01} that $K_p$ embeds in $\Sigma$ where $p\geq \sqrt{6g+9}$. Let $T$ be the rooted tree such that the root vertex has degree $\Delta-p+1$, every non-root non-leaf vertex has  degree $\Delta$, and the distance between the root and each leaf  equals $(k-1)/2$. Observe that $T$ has $(\Delta-p+1)(\Delta-1)^{(k-3)/2}$ leaf vertices.  For each vertex $v$ of $K_p$ take a copy of $T$ and identify the root of $T$ with $v$. The obtained graph embeds in $\Sigma$, has maximum degree $\Delta$, and has diameter $k$. The number of vertices is at least $p(\Delta-p+1)(\Delta-1)^{(k-3)/2}$. It follows that for odd $k$,  for all $\epsilon>0$ and sufficiently large $\Delta\geq\Delta(g,\epsilon)$, 
\begin{equation}
\label{LowerBound}
N(\Delta,k,\mathcal{G}_\Sigma)\geq(1-\epsilon)\sqrt{6g+9}\,(\Delta-1)^{(k-1)/2}.
\end{equation}
This lower bound is within a $O(g)$ factor of the upper bound in Theorem~\ref{theo:GralSurface}. Moreover, combined with the above upper bound for planar graphs, this result solves an open problem by Miller and \v{S}ira\v{n} \cite[Prob.~13]{MS05a}. They asked whether Knor and \v{S}ir\'a\v{n}'s result could be generalised as follows: is it true that, for each surface $\Sigma$ and for each diameter $k\ge 2$, there exists  $\Delta_0:=\Delta_0(\Sigma,k)$ such that $N(\Delta,k,\mathcal{G}_\Sigma)=N(\Delta,k,\mathcal{P})$ for $\Delta\ge \Delta_0$?   
We now give a negative answer to this question for odd $k$. 
Equation~\eqref{LowerBound} says that $N (\Delta, k, \mathcal{G}_\Sigma)/(\Delta-1)^{\lfloor{k/2}\rfloor}\ge c\sqrt{g}+1$, while Theorem~\ref{theo:GralSurface} with $g=0$ says that $N (\Delta, k, \mathcal{P})/(\Delta-1)^{\lfloor{k/2}\rfloor}\le c'$, for absolute constants $c$ and $c'$. Thus $N(\Delta,k,\mathcal{G}_\Sigma)> N(\Delta, k, \mathcal{P})$ for odd $k\ge 3$ and $g$ greater than some absolute constant.

In the literature all upper bounds for $N(\Delta,k,\mathcal{P})$ or $N(\Delta,k,\mathcal{G}_\Sigma)$ rely on graph separator theorems. Fellows {\it et al.} \cite[Cor.~14]{FHS95} used the graph separator theorem for planar graphs by Lipton and Tarjan \cite[Lem.~2]{LT79}, while Tishchenko used an extension of Lipton and Tarjan's theorem proved by himself in \cite[Cor.~3.3]{Tis2011a}. In the same vein, \v{S}iagiov\'a and Simanjuntak \cite{SR04} made use of Djidjev's separator theorem \cite[Lem.~3]{Dji85} for graphs on surfaces. Our proofs rely on a new graph separator theorem, also proved  in this paper, which extends Tischenko's separator theorem to all surfaces, and is of independent interest.   

In this paper we follow  the notation and terminology of \cite{Die05}. The remainder of the paper is organised as follows. Section~\ref{sec:l-separators} proves a separator theorem for graphs on surfaces. Section~\ref{sec:surface-graphs} is devoted to the proof of Theorem~\ref{theo:GralSurface}.  
Finally, Section~\ref{sec:conclusion} discusses some open problems arising as a result of our work.

\section{$\ell$-Separators in multigraphs on surfaces} 
\label{sec:l-separators}

A \emph{triangulation} of a
surface $\Sigma$ is a multigraph (without loops) embedded in $\Sigma$ such
that each face is bounded by exactly 3 edges. Let $\ell\in \mathbb{Z}^+$ and let $\Sigma$ be a surface of Euler genus $g$ and let $G$ be an $n$-vertex triangulation of $\Sigma$. The aim of this section is to find a ``small'' subgraph $S$ of $G$ with $\ell$ faces such that each face of $S$ contains ``many'' vertices of $G$.

A well-known result by Lipton and Tarjan {\cite[Lem.~2]{LT79} states that if $\ell=2$ then there exists a subgraph $S$ of order at most $(\ell-1)(2r+1)$ in every plane triangulation $G$ such that each face of $S$ contains at least $\frac{n}{2\ell-1}-|S|$ vertices of $G$. Here $r$ denotes the radius of $G$. Tishchenko \cite[Thm.~1.1,Thm.~1.2]{Tis2011} found such a subgraph $S$ in a  plane triangulation for every $\ell\ge2$. Tishchenko \cite{Tis2011} called such subgraphs {\it $\ell$-separators} by virtue of its number of faces. Our result extends Tishchenko's result to all surfaces.

A \emph{tree decomposition} of a multigraph $G$ is a pair $(T,\{B_z:z\in V(T)\})$ consisting of a tree $T$ and a collection of sets of vertices in $G$ (called  \emph{bags}) indexed by the nodes of $T$, such that:
\begin{enumerate}
\item $\bigcup\{B_z:z\in V(T)\}=V(G)$, and
\item for every edge $vw$ of $G$, some bag $B_z$ contains both $v$ and $w$, and 
\item for every vertex $v$ of $G$, the set $\{z\in V(T):v\in B_z\}$ induces a non-empty (connected) subtree of $T$.
\end{enumerate}
For a subtree $Q$ of $T$, let $G[Q]$ be the subgraph of $G$ induced by $$\bigcup\big\{B_z:z\in V(Q)\big\}\setminus\bigcup\big\{B_z:z\in V(T)\setminus V(Q)\big\}.$$
Thus a vertex $v$ of $G$ is in $G[Q]$ whenever $v$ is in some bag in $Q$ and is in no bag outside of $Q$. 

Our approach to finding an $\ell$-separator in an embedded multigraph is based on the following lemma for finding a separator in a multigraph with a given tree decomposition.
 
\begin{lem} 
\label{lem:TreeDecompSeparator}
Let $\ell\geq0$ and $b\geq2$ be integers. Let $G$ be a  multigraph with $n\geq (3\ell+1) b$ vertices. Let $(T,\{ B_z : z \in V(T) \})$ be a  tree decomposition of $G$, such that $T$ has maximum degree at most 3, and $|B_z|\leq b$ for each $z\in V(T)$. Then there is a set $R$ of exactly $\ell$ edges of $T$ such that for each of the $\ell+1$ components $Q$ of $T-R$, 
$$|G[Q]|\geq \frac{n-\ell b}{2\ell+1}.$$
\end{lem}

\begin{proof} We proceed by induction on $\ell\geq0$. The base case with $\ell=0$ and $R=\emptyset$ is trivially true. Now assume that $\ell\geq1$. Observe that $|E(T)|\ge \ell$ since $n\ge (3\ell+1)b$ and each bag has size at most $b$. 

Consider an edge $xy$ of $T$. Let $T(x,y)$ and $T(y,x)$ be the subtrees of $T$ obtained by deleting the edge $xy$, where $T(x,y)$ contains $x$ and $T(y,x)$ contains $y$. 
Let $G(x,y):=G[T(x,y)]$ and $G(y,x):=G[T(y,x)]$. 
By part (3) of the definition of tree decomposition,  each vertex of $G$ is in either $G(x,y)$ or $G(y,x)$ or $B_x\cap B_y$. Orient each edge $xy$ of $T$ by $\overrightarrow{xy}$ if $$|G(x,y)|<\frac{n-\ell b}{2\ell+1}.$$

{\bf \noindent  Case 1.} Some edge $xy\in E(T)$ is oriented in both directions: Then $|G(x,y)|<\frac{n-\ell b}{2\ell+1}$ and 
$|G(y,x)|<\frac{n-\ell b}{2\ell+1}$. Thus $$n=|G(x,y)|+|G(y,x)|+|B_x\cap B_y|< 2\left(\frac{n-\ell b}{2\ell+1}\right)+b.$$ Hence
$n(2\ell+1)< 2(n-\ell b)+b(2\ell+1)=2n+b$ and
$n(2\ell-1)< b$, which is a contradiction. 

Now assume that each edge is oriented in at most one direction. A vertex $x$ of $T$ is a \emph{sink} if no edge incident with $x$ is oriented away from $x$. (Note that some edges incident with a sink might be unoriented.)\ Let $J$ be the subforest of $T$ obtained as follows: every sink is in $J$, and if $xy$ is an unoriented edge incident with a sink $x$, then $y$ and $xy$ are in $J$. Note that the vertex $y$ is also a sink and so every vertex in $J$ is a sink. Since $T$ is acyclic, $V(J)\neq\emptyset$. 

{\bf \noindent  Case 2.} $E(J)=\emptyset$: Thus $J$ contains an isolated vertex $y$. 
Let $x_1,\dots,x_d$ be the neighbours of $y$, where $d\leq 3$. Since $y$ is a sink and is isolated in $J$, each edge $x_iy$ is oriented $\overrightarrow{x_iy}$. Thus $|G(x_i,y)|< \frac{n-\ell b}{2\ell+1}$. Every vertex not in $\bigcup_i G(x_i,y)$ is in $B_y$. Thus $$n\leq b+\sum_i|G(x_i,y)|< b+3\left(\frac{n-\ell b}{2\ell+1}\right).$$ 
Thus $n(2\ell+1)< b(2\ell+1)+3(n-\ell b)=3n-\ell b + b$ and $0\leq n(2\ell-2)< b(1-\ell)\leq 0$, which is a contradiction.

{\bf \noindent Case 3.} $E(J)\neq\emptyset$: Let $x$ be a leaf vertex in $J$. Thus $x$ is a sink  and is incident with exactly one unoriented edge $xy$. Let $x_1,\dots,x_d$ be the other neighbours of $x$ in $T$, where $d\leq 2$. Thus $x_ix$ is oriented $\overrightarrow{x_ix}$. 
Let $T':=T(y,x)$ and $G':=G(y,x)$ and $n':=|G'|$. Then $(T',\{B_z\setminus(B_x\cap B_y):z\in V(T')\})$ is a tree-decomposition of $G'$. 
Since $x_ix$ is oriented $\overrightarrow{x_ix}$, 
$$n\leq |B_x|+ n' + \sum_i|G(x_i,x)| \leq b+n'+ 2\left(\frac{n-\ell b}{2\ell+1}\right).$$ 
It follows that 
\begin{align*}n' \geq \frac{(2\ell-1)n-b}{2\ell+1}\ge \frac{(2\ell-1)(3\ell+1)b-b}{2\ell+1}=(3(\ell-1)+1)b.\end{align*}
 By induction, there is a set $R'$ of $\ell-1$ edges of $T'$ such that for each component $Q'$ of $T'-R'$, 
\begin{align*}
|G'[Q']| \geq  \frac{n'-(\ell-1)b}{2\ell-1}.
\end{align*}

We now prove that $R:=R'\cup\{xy\}$ satisfies the lemma. By definition, $|R|=\ell$. Each component of $T-R$ is either $T(x,y)$ or is a component of $T'-R'$. Since $xy$ is unoriented, $|G(x,y)|\geq\frac{n-\ell b}{2\ell+1}$, as required. For each component $Q'$ of $T'-R'$,
\begin{align*}
|G[Q']|  \,=\, |G'[Q']| 
\,\geq\,  \frac{n'-(\ell-1)b}{2\ell-1}
\,\geq\,  \frac{n}{2\ell+1} - \frac{b}{(2\ell+1)(2\ell-1)}-\frac{(\ell-1)b}{2\ell-1}
\,=\,   \frac{n-\ell b}{2\ell+1},
\end{align*}
as required. Hence $R$ satisfies the lemma.
\end{proof}

\begin{thm} 
\label{thm:SeparatorSurface}
Let $\ell \in \mathbb{Z}^+$. Let $\Sigma$ be a surface with Euler genus $g$. Let $G$ be a triangulation of $\Sigma$ with radius $r$ and order $n\ge(3\ell+1)((3+2g)r+1)$. Then $G$ has a subgraph $S$ with at most $(2r+1)(g+\ell)$ edges, such that the induced embedding of $S$ in $\Sigma$ is 2-cell with $\ell+1$ faces, and each face of $S$ contains at least $$\frac{n-\ell (3+2g)r-\ell}{2\ell+1}$$ vertices of $G$ in its interior. 
\end{thm}

\begin{proof}
Let $u$ be a centre of $G$. Let $T$ be a breadth-first spanning tree of $G$ rooted at $u$. Thus $\dist_T(u,v)=\dist_G(u,v)\leq r$ for each vertex $v$ of $G$. Let $T_v$ be the $uv$-path in $T$. 

Various authors \cite{Big71,RicSha84,Sko92} proved that there is a set $X$ of  exactly $g$ edges in $G-E(T)$ such that the induced embedding of $T\cup X$ in $\Sigma$ is 2-cell and has exactly one face. Let $F(G)$ be the set of faces of $G$. If $T^*$ is the graph with vertex set $F(G)$, where faces $f_1$ and $f_2$ of $G$ are adjacent in $T^*$ whenever $f_1$ and $f_2$ share an edge in $E(G)\setminus(E(T)\cup X)$, then $T^*$ is a tree with maximum degree at most 3. For each face $f=xyz$ of $G$, let  
$$B_f:=V(T_x\cup T_y\cup T_z)\,\cup\,\bigcup_{pq\in X}V(T_p\cup T_q).$$ Dujmovi\'c et al. \cite[Thm.~7]{DujMorWoo13} proved that $(T^*,\{B_f:f\in V(T^*)\})$ is a tree decomposition of $G$. Clearly, $T^*$ has maximum degree at most 3, and $|B_f|\leq (3+2g)r+1$ for each $f\in V(T^*)$ (since each $T_v$ has at most $r+1$ vertices, one of which is $u$). 

By Lemma~\ref{lem:TreeDecompSeparator} with $b=(3+2g)r+1$, there is a set $R$ of $\ell$ edges of $T^*$ such that 
$|G[Q]| \geq \frac{n- \ell (3+2g)r-\ell}{2\ell+1}$ for each of the $\ell+1$ components $Q$ of $T^*-R$.
Let $L$ be the set of edges $vw$ of $G$, such that for some edge $f_1f_2$ of $T^*$ in $R$, we have that $vw$ is the common edge on the faces $f_1$ and $f_2$ in $E(G)\setminus(E(T)\cup X)$. Thus $|L|=|R|=\ell$. 

For each edge $vw$ of $G-E(T)$, let $Y_{vw}:=T_v\cup T_w\bigcup\{vw\}$. Note that $Y_{vw}$ has at most $2r+1$ edges. Let $S:=\bigcup\{Y_{vw}:vw\in X\cup L\}$. Thus $S$ has at most $(2r+1)(g+\ell)$ edges. Starting from the 2-cell embedding of $T\cup X$ with one face, the addition of each edge in $L$ splits one face into two, giving $\ell+1$ faces in total. Thus $S$, which is obtained from $T\cup X\cup L$ by deleting pendant subtrees, also has $\ell+1$ faces, and is 2-cell embedded. 

The faces of $S$ are in 1--1 correspondence with the components of $T^*-R$. Let $\Phi$ be the face of $S$ corresponding to some component $Q$ of $T^*-R$.  Let $v$ be one of the at least $\frac{n- \ell (3+2g)r-\ell}{2\ell+1}$ vertices in $G[Q]$. 
If $v$ is not strictly in the interior of $\Phi$, then $v\in B_f$, where $f$ is a face of $G$ that is outside of $\Phi$ and incident with $v$, contradicting that $v$ is in $G[Q]$. 
Hence each face of $S$ contains at least $\frac{n- \ell (3+2g)r-\ell}{2\ell+1}$ vertices in its interior. 
\end{proof}

The case of planar graphs is worth particular mention, and is similar to a result by Tishchenko \cite[Cor.~33]{Tis2011a}.

\begin{cor} 
\label{cor:SeparatorPlanar}
Let $\ell \in \mathbb{Z}^+$. Let $G$ be a triangulation of the sphere with radius $r$ and order $n\ge(3\ell+1)(3r+1)$.  Then $G$ has a subgraph $S$ with at most $\ell(2r+1)$ edges, such that the induced embedding of $S$ is 2-cell with $\ell+1$ faces, and each face of $S$ contains at least $$\frac{n-(3r+1)\ell}{2\ell+1}$$ vertices of $G$ in its interior. 
\end{cor}

\section{Proof of Theorem~\ref{theo:GralSurface}} 
\label{sec:surface-graphs}

We start the section with a well-known lemma.

\begin{lem}[Euler's formula, {\cite[pp.~95]{MohTho01}}]\label{lem:EulerFormula}
Let $G$ be a multigraph which is embedded in a surface $\Sigma$ of Euler genus $g$. Then
\[|V(G)|-|E(G)|+|F(G)| \ge 2-g,\]
where $V(G)$, $E(G)$, and $F(G)$ denote the set of vertices, edges, and faces of $G$, respectively. Equality is achieved when the multigraph  embeds 2-cellularly in $\Sigma$.
\end{lem}  

Let  $S$ be a connected multigraph with minimum degree at least 2 and maximum degree at least 3 which is embedded in a surface $\Sigma$. We define a multigraph $H$ from $S$ as follows: if there is an edge $e$ with a degree-2 endvertex then contract $e$, and repeat until the minimum degree is at least $3$. The multigraph so constructed is called the {\it simplified configuration} of $S$ \cite{Tis2011}. During the edge contraction we do not allow a facial walk to vanish; that is, a facial walk can become a loop but not a point. Note that any two sequences of edge contractions result in isomorphic multigraphs and that $H$ could also be defined as the minimal multigraph such that $S$ is a subdivision of H. We call a vertex of $S$ or $H$ a {\it branch vertex} if it has degree at least three in $S$ or $H$, respectively; every vertex in $H$ is a branch vertex. Also, $H$ may have faces of length 1 (the loops) and faces of length 2, and it is connected.  See Fig.~\ref{fig:Goodregions} for an example. 

\begin{figure}[!ht]
\begin{center}
\includegraphics[scale=0.9]{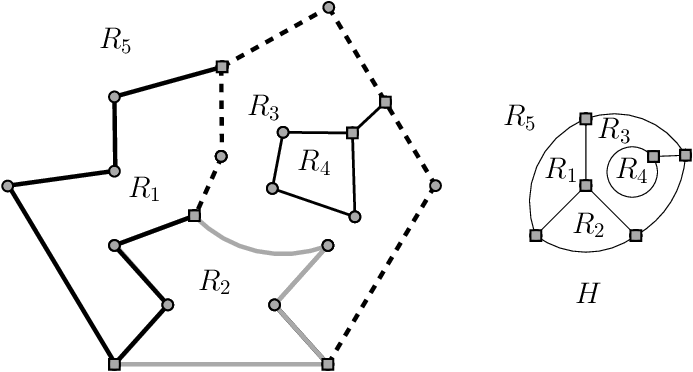}
\caption{($a$) A 5-separator in the plane.  ($b$) The associated simplified configuration $H$. Branch vertices are represented by a square.}
\label{fig:Goodregions}
\end{center}
\end{figure}

Our  Theorem~\ref{theo:GralSurface} follows from the following technical result. 

\begin{thm}\label{theo:mainSurface} Let $G$ be a graph embeddable in a surface with Euler genus at most $g$, maximum degree $\Delta\geq 3$, and diameter
$k\geq 2$. Then 
\[|V(G)|< (2\ell+1)c(\Delta-1)^{\left\lfloor{k}/{2}\right\rfloor}+(2\ell+1)(2k+1)(g+\ell)M+\ell(3+2g)k+\ell,\]
where $M=M(\Delta,\lfloor k/2\rfloor-1)$ denotes the corresponding Moore bound and \[(\ell,c):=\begin{cases}(\ceil{g^{2/3} + g^{1/2}}+6,2 g^{1/3}+ 6)&\text{if $k$ is even}\\
(\ceil{\sqrt{42g}}+33, 2\ell+2g-1)&\text{if $k$ is odd}.\end{cases}\]
\end{thm}

Note that  the assumed lower bounds on $\Delta$ in Theorem~\ref{theo:GralSurface} ensure that the secondary term $(2\ell+1)(2k+1)(g+\ell)M+\ell(3+2g)k+\ell$ in the upper bound on $N(\Delta,k,\mathcal{G}_\Sigma)$ in Theorem {\ref{theo:mainSurface} is not dominant. 

\begin{proof}[Proof of Theorem {\ref{theo:mainSurface}}]
 By \cite[Prop.~3.4.1, Prop.~3.4.2]{MohTho01}, we may
assume  that $G$ is 2-cell embedded in a surface $\Sigma$ of Euler genus $g$. Suppose for the sake of contradiction that \begin{equation}\label{eq:order}|V(G)|\ge (2\ell+1)c(\Delta-1)^{\left\lfloor{k}/{2}\right\rfloor}+(2\ell+1)(2k+1)(g+\ell)M+\ell(3+2g)k+\ell.\end{equation}

It follows that $|V(G)|\ge (3\ell+1)((3+2g)k+1)$. Thus, we may
apply Theorem~\ref{thm:SeparatorSurface} to a triangulation $G'$ of $G$. Note that $G'$ may be a multigraph. Let $S$ be a subgraph of $G'$ satisfying  Theorem~\ref{thm:SeparatorSurface}. Thus $|E(S)|\le (2k+1)(g+\ell)$, and the induced embedding of $S$ in $\Sigma$ has exactly $\ell+1$ faces $R_1,\ldots,R_{\ell+1}$ such that 
\begin{equation}\label{eq:TheoFunCycSep}
|V(G)\cap R_i|\ge \frac{|V(G)|}{2\ell+1}-\frac{\ell(3+2g)k+\ell}{2\ell+1},\; \text{for $i\in [1,\ell+1]$.}
\end{equation}

 For each face $R_i$ of $S$, let $\partial(R_i)$ be the subgraph of $S$ consisting of the vertices and edges embedded in the boundary of $R_i$. A vertex in $V(G)\cap R_i$ is {\it deep} if it is at distance at least $\lfloor k/2\rfloor$ in $G$ from $\partial(R_i)$.

The rest of the proof proceeds as follows. We first give a lower bound of $c(\Delta-1)^{\lfloor{k}/{2}\rfloor}$ for the number of deep vertices $D_i$ in each face $R_i$ of $S$. This implies that for every pair of distinct faces $R_i$ and $R_j$ of $S$ either $\partial(R_i)$ and $\partial(R_j)$ intersect or there exists an edge of $G$ with an endvertex in $\partial(R_i)$ and another endvertex on $\partial(R_j)$. Then we show that the embedding of $G$ restricts the number of pairs of faces of $S$ whose boundaries share an edge; these are our good pairs of faces. We bound the number of good pairs by a function linear in $\ell$. It follows that the number of pairs of faces of $S$ whose boundaries do not share an edge is quadratic in $\ell$; these are our bad pairs of faces. 

If the diameter is even  the set $I_{ij}$ of vertices in $\partial(R_i)\cap \partial(R_j)$ is nonempty for each bad  pair of regions $R_i$ and $R_j$. Furthermore, for any pair of deep vertices $x\in D_i$  and $y\in D_j$ every $xy$-path of length at most $k$ includes some vertex in $I_{ij}$. This allows us to provide an upper bound for the number of deep vertices in certain regions $R_i$ and $R_{j*}$. For our selection of $\ell$ and $c$ this upper bound turns out to be smaller than the aforementioned lower bound of $2c(\Delta-1)^{\lfloor{k}/{2}\rfloor}$ for $|D_i|+|D_{j*}|$, giving the desired contradiction.  In the case of odd diameter some $xy$-paths between deep vertices $x\in D_i$  and $y\in D_j$ may avoid $I_{ij}$, forcing the existence of edges between the boundaries of the bad pair of faces; these are our jump edges. The proof ends when we show that the necessary quadratic (in $\ell$) number of jump edges is inconsistent with a surface embedding. 

In the following we detail these ideas formally.

Let $V_i:=V(G)\cap R_i$ and let $D_i$ be the set of deep vertices in $R_i$. Since $\partial(R_i)$ has at most $(2k+1)(g+\ell)$ vertices and since the number of vertices at distance at most $\floor{k/2}-1$ from a given vertex is at most $M(\Delta,\floor{k/2}-1)$,
\[|V_i|\le (2k+1)(g+\ell)M+|D_i|.\] 
By  (\ref{eq:order}) and (\ref{eq:TheoFunCycSep}),
\begin{align*}|V_i|\ge c(\Delta-1)^{\left\lfloor {k}/{2}\right\rfloor}+(2k+1)(g+\ell)M+\frac{\ell(3+2g)k+\ell}{2\ell+1}-\frac{\ell(3+2g)k+\ell}{2\ell+1}.\end{align*}  
Thus, $c(\Delta-1)^{\left\lfloor k/2\right\rfloor}+(2k+1)(g+\ell)M\le (2k+1)(g+\ell)M+|D_i|$, implying
\begin{equation}
\label{eq:DeepVertSurf}  |D_i|\ge c(\Delta-1)^{\left\lfloor {k}/{2}\right\rfloor}.
\end{equation}

Let $H$ be the simplified configuration of $S$. Since $S$ is a connected multigraph with at least $3$ faces, $S$ has minimum degree at least 2 and maximum degree at least 3. The multigraph $H$ has minimum degree at least 3 and $\ell+1$ faces, and it may include faces of length 1 or 2. It is connected and embeds 2-cellularly in $\Sigma$. We use the multigraph $H$ to count the branch vertices of $S$. Since $3|V(H)|\le 2|E(H)|$, Lemma~\ref{lem:EulerFormula} gives
\begin{align}\label{eq:SurfaceH}
|V(H)|\le 2\ell+2g-2\; \text{and}\;|E(H)|&\le 3\ell+3g-3. 
\end{align}

Distinct faces $R_i$ and $R_j$ of $S$ are a {\it good} pair if their boundaries share an edge in $S$; otherwise they are a {\it bad} pair.

Since the number of good pairs of faces of $S$ is at most $|E(H)|$, the number of bad pairs of faces of $S$ is at least 
\begin{align}
\label{eq3} {\ell+1 \choose 2}-(3\ell+3g-3).
\end{align}

 Let $R_i$ and $R_j$ ($i\ne j$) be a bad pair of faces of $S$.  Let $I_{ij}$ be the set of vertices in $\partial(R_i)\cap \partial(R_j)$. 
 
We first prove the theorem for even $k$. Note that $I_{ij}\ne \emptyset$ for each bad pair of regions, since $D_i$ and $D_j$ are nonempty. For each $i$, let $\ell_i$ be the number of bad pairs in which $R_i$ is involved.  Choose $i$ so that $\ell_i$ is maximum, then $\ell_i\ge 2\frac{{\ell+1 \choose 2}-(3\ell+3g-3)}{\ell+1}=\frac{\ell^2-5\ell-6g+6}{\ell+1}\ge 1$, since $\ell=\ceil{g^{2/3} + g^{1/2}}+6$. For simplicity of notation, assume the faces $R_1,\ldots,R_{\ell_i}$ are involved in those pairs, and $i\not\in\{1,\ldots,\ell_i\}$. 
  
Let $\Lambda$ be the multigraph formed from $\cup_{j=1}^{\ell_i} \partial(R_j)\cup \partial(R_i)$ by contracting each edge not incident to two vertices of $\cup_{j=1}^{\ell_i} I_{ij}$; see Fig.~\ref{fig:IntSurf}. Thus $\Lambda$ has vertex set $\cup_{j=1}^{\ell_i} I_{ij}$ and edge set formed by the edges left after the contractions. 
\begin{figure}[!ht]
\begin{center}
\includegraphics[scale=1]{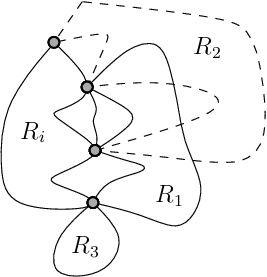}
\caption{A possible configuration for the multigraph $\Lambda$.}
\label{fig:IntSurf}
\end{center}
\end{figure}

Since $|F(\Lambda)|\le |F(S)|=\ell+1$ and since every vertex of $\Lambda$ has degree at least 4, Lemma~\ref{lem:EulerFormula} gives 
\begin{align}\label{eq:Iij-1}
|V(\Lambda)|=|\cup_{j=1}^{\ell_i} I_{ij}|\le \ell+g-1\quad \text{and} \quad|E(\Lambda)|\le 2(\ell+g-1). 
\end{align}

Each face $R_j$ ($j\in \{1,\ldots,\ell_i\}$) of $\Lambda$ has $|I_{ij}|$ vertices, and thus has $|I_{ij}|$ edges. Each such edge is in at most two such faces. Furthermore,  the face $R_i$ of $\Lambda$ has $|V(\Lambda)|$ edges and shares no edge with a face $R_j$ ($j\in \{1,\ldots,\ell_i\}$). Thus 
\begin{align}\label{eq:Iij-2}
2|E(\Lambda)| \geq 2|V(\Lambda)| +\sum_{j=1}^{\ell_i} |I_{ij}| \geq 2|V(\Lambda)|+\ell_i |I_{ij*}|,
\end{align}
where $I_{ij*}$  is a set $I_{ij}$ of minimum size.

Combining (\ref{eq:Iij-1}) and (\ref{eq:Iij-2}), 
\begin{align*}
|I_{ij*}|&\le\frac{4(\ell+g-1)-2|V(\Lambda)|}{\ell_i}\le \frac{4(\ell+g-1)}{\ell_i}-\frac{2|I_{ij*}|}{{\ell_i}} \quad\text{(since $|V(\Lambda)|\ge |I_{ij*}|$)},\\
|I_{ij*}|\left(1+\frac{2}{\ell_i}\right)&\le\frac{4(\ell+g-1)}{\ell_i},\\
|I_{ij*}|&\le\frac{4(\ell+g-1)}{\ell_i+2}\le \frac{4(\ell+g-1)(\ell+1)}{\ell^2-3\ell-6g+8}\quad\text{(since $\ell_i\ge \tfrac{\ell^2-5\ell-6g+6}{\ell+1}$).}\end{align*}

For $x\in D_i$ and $y\in D_{j*}$ every $xy$-path of length $k$ includes some vertex in $I_{ij*}$. Thus every vertex in $D_i\cup D_{j*}$  is at distance $k/2$ from $I_{ij*}$. Since the number of vertices at distance $t$ from a fixed vertex  is at most $(\Delta-1)^t$, by (\ref{eq:DeepVertSurf}), \[2c(\Delta-1)^{{k}/{2}}\le |D_i|+|D_{j*}|\le|I_{ij*}|(\Delta-1)^{{k}/{2}}\le \frac{4(\ell+g-1)(\ell+1)}{\ell^2-3\ell-6g+8}(\Delta-1)^{{k}/{2}},\]  
which is a contradiction for $\ell=\ceil{g^{2/3} + g^{1/2}}+6$ and $c=2 g^{1/3}+ 6$.

Now assume that $k$ is odd. Consider any two faces $R_i$ and $R_j$ of $S$, then an edge $xy$ in $G$ with $x\in \partial(R_i)-\partial(R_j)$ and $y\in \partial(R_j)-\partial(R_i)$ is called a {\it jump edge between} $R_i$ and $R_j$. We say that two jump edges are {\it equivalent} if they connect the same set of pairs of faces. 

\smallskip
\noindent {\bf Case 1:} There is no jump edge between some bad pair of faces $R_i$ and $R_j$. 

We follow the reasoning of the even case. Let $\Lambda$ be the multigraph formed from $\partial R_i\cup \partial R_j$ by contracting each edge not incident to two vertices of $I_{ij}$. Thus $\Lambda$ has vertex set $I_{ij}$ and edge set formed by the edges left after the contractions. Since $D_i\ne \emptyset$ and $D_j\ne \emptyset$ and since there is no jump edge between $R_i$ and $R_j$, we must have $I_{ij}\ne \emptyset$. It follows that $|F(\Lambda)|\le |F(S)|=\ell+1$ and that the minimum degree of $\Lambda$ is at least 4. Thus, by Lemma~\ref{lem:EulerFormula}, $|V(\Lambda)|\le \ell+g-1$.

For $x\in D_i$ and $y\in D_j$, since $\text{dist}(x, \partial(R_i))\ge \lfloor k/2\rfloor$ and $\text{dist}(y, \partial(R_j))\ge \lfloor k/2\rfloor$ and because there is no jump edge between $R_i$ and $R_j$, every $xy$-path of length at most $k$ includes some vertex in $I_{ij}$. If $\text{dist}(x, I_{ij})\ge \lfloor k/2\rfloor+1$ and $\text{dist}(y, I_{ij})\ge \lfloor k/2\rfloor+1$ for some $x\in D_i$ and $y\in D_j$, then $\text{dist}(x,y)\ge k+1$. Thus, without loss of generality, every vertex in $D_i$ is at distance exactly $\lfloor k/2\rfloor$ from $I_{ij}$. By (\ref{eq:DeepVertSurf}),
\[c(\Delta-1)^{\left\lfloor {k}/{2}\right\rfloor}\le|D_i|\le|I_{ij}|(\Delta-1)^{\left\lfloor {k}/{2}\right\rfloor}\le(\ell+g-1)(\Delta-1)^{\left\lfloor {k}/{2}\right\rfloor},\]
which is a contradiction since $c=2\ell+2g-1$. 

\smallskip
\noindent {\bf Case 2:} Now assume that between every bad pair of faces there is a jump edge. 

 A jump edge $xy$ is {\it normal} if neither $x$ nor $y$ is a branch vertex in $S$, otherwise it is {\it special}. Observe that a normal jump edge connects exactly one pair of regions. (This is not true for special jump edges.) 

Let $X$ be the multigraph consisting of $S$ plus the jump edges. The multigraph $X$ is connected and  may have more than $\ell+1$ faces. Now define a multigraph $Y$ obtained from $X$ by contracting an edge whenever it is not a jump edge and no endvertex is  a branch vertex of $S$. During the edge contraction we do not allow the facial walk of a face to vanish; that is, a facial walk may become a loop but not a point.  See Fig.~\ref{fig:EdgeCont}. Also, a set of jump edges running (in ``parallel'') between the same set of pairs of regions are replaced by a single edge. 

\begin{figure}[!ht]
\begin{center}
\includegraphics[scale=1]{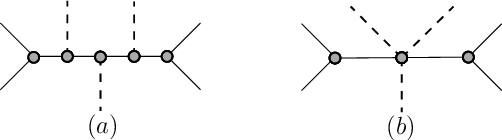}
\caption{Contraction of edges which are ends of jump edges. ($a$) various jump edge ends (in dashed lines). ($b$) The resulting edges in Y.}
\label{fig:EdgeCont}
\end{center}
\end{figure}

Observe that $Y$ can be obtained from a subdivision of $H$ by adding the jump edges, where each edge of $H$ is subdivided at most once. Thus \begin{align*}
|V(Y)|\le & |E(H)|+|V(H)|\le 5\ell+5g-5.
 \end{align*}

The multigraph $Y$ may have faces of length 1 or 2, and it is connected and of minimum degree at least 3. 
  
 Denote by $F_{1}(Y)$ and $F_{2}(Y)$ the set of faces of $Y$ of length 1 or 2, respectively. Then $|F_{1}(Y)|+|F_{2}(Y)|\le \ell+1$; this is the case because $Y$ has no multiple jump edges, and therefore, faces of length 1 and 2 can only arise from the initial faces of $H$. The handshaking lemma for faces gives $3(|F(Y)|-|F_1(Y)|-|F_2(Y)|)+|F_1(Y)|+2|F_2(Y)|\le 2|E(Y)|$. Thus, 
\begin{align*}
|F(Y)|&\le \tfrac{2}{3}|E(Y)|+\tfrac{1}{3}(2|F_{1}(Y)|+|F_{2}(Y)|)\le \tfrac{2}{3}E(Y)|+\tfrac{2}{3}(\ell+1).
\end{align*}
Consequently, Lemma~\ref{lem:EulerFormula} gives that
\begin{align*}
|E(Y)|\le 3|V(Y)|+2\ell-4+3g& \le 
 17\ell+18g-19.
 \end{align*}

The number of bad pairs  of faces of $S$ that are connected by normal jump edges equals the number of normal jump edges, and hence, is at most $|E(Y)|$.
Since the number of bad pairs of faces of $S$ is at least ${\ell+1 \choose 2}-(3\ell+3g-3)$ and since $\ell=\ceil{\sqrt{42g}}+33$, the number of bad pairs of faces of $S$ that are not joined  by a normal jump edge is at least \begin{equation}\label{eq5}{\ell+1\choose 2}-(3\ell+3g-3)-(17\ell+18g-19)\ge 1.\end{equation}
Hence, there is at least one bad pair of faces $R_i$ and $R_j$ of $S$ that is not joined by a normal jump edge.
 
Recall the number of branch vertices in $S$ equals the number of vertices of $H$, which is at most $2\ell+2g-2$. Thus, the number of deep vertices in each of $D_i$ and $D_j$ at distance $
\lfloor k/2\rfloor$ from a branch vertex in $\partial(R_i)\cup \partial(R_j)$ is at most 
\[(2\ell+2g-2)(\Delta-1)^{\left\lfloor {k}/{2}\right\rfloor}.\]   

By Equation (\ref{eq:DeepVertSurf}), $|D_i|\ge c(\Delta-1)^{\floor{k/2}}$ and $|D_j|\ge c(\Delta-1)^{\floor{k/2}}$. Since $c>2\ell+2g-2$ there are vertices $\beta_i$ and $\beta_j$ in $D_i$ and $D_j$ respectively at distance at least $\lfloor k/2\rfloor+1$ from each branch vertex of $H$. Thus a shortest path of length $k$ between $\beta_i$ and $\beta_j$ must use a normal jump edge between $R_i$ and $R_j$. This is a contradiction and completes the proof of the theorem.
\end{proof}

\section{Concluding remarks} 
\label{sec:conclusion}

We believe the asymptotic value of $N(\Delta,k,\mathcal{G}_\Sigma)$ is closer to the lower bound in 
Equation~\ref{LowerBound}
than to the upper bound in Theorem~\ref{theo:GralSurface}.

\begin{conjecture}\label{conj1} 

There exist a constant $c$ and a function $\Delta_0:=\Delta_0(g,k)$ such that, for $\Delta\ge \Delta_0$, 
\[N(\Delta,k,\mathcal{G}_\Sigma)\le\begin{cases}
c(\Delta-1)^{\lfloor {k}/{2}\rfloor}& \text{if $k$ is even}\\
c(\sqrt{g}+1)(\Delta-1)^{\lfloor {k}/{2}\rfloor}& \text{if $k$ is odd.}
\end{cases}\]
\end{conjecture}

A generalisation to the class $\mathcal{G}_H$ of $H$-minor-free graphs, with $H$ a fixed graph, was studied in {\cite{PVW12}}. The current best upper bound of $$N(\Delta,k,\mathcal{G}_H)\le 4k(c|H|\sqrt{\log{|H|}})^k\Delta^{\lfloor k/2\rfloor}$$ was given in {\cite[Sec.~4]{PVW12}}. Note that if $H$ is planar, then $\mathcal{G}_H$ has bounded treewidth, and thus has small order {\cite[Thm.~12]{PVW12}}.

\section*{Acknowledgments}

We would like to thank Roi Krakovski for discussions that led to improvements in the paper presentation. 

Research of Nevo was partially supported by the Marie Curie grant IRG-270923 and the ISF grant 805/11. Research of Pineda-Villavicencio was supported by a postdoctoral fellowship funded by the Skirball Foundation via the Center for Advanced Studies in Mathematics at Ben-Gurion University of the Negev, and by an ISF grant. Research of Wood is supported by the Australian Research Council.

\providecommand{\bysame}{\leavevmode\hbox to3em{\hrulefill}\thinspace}
\providecommand{\MR}{\relax\ifhmode\unskip\space\fi MR }
\providecommand{\MRhref}[2]{%
  \href{http://www.ams.org/mathscinet-getitem?mr=#1}{#2}
}
\providecommand{\href}[2]{#2}
\bibliographystyle{plainnat}

\end{document}